\documentclass[a4paper]{article}

\usepackage{amsthm}
\usepackage{amsmath}
\usepackage{amsfonts}
\usepackage{amssymb}
\usepackage{amsthm}

\usepackage{tikz-cd}
\usepackage{bookmark}
\usepackage{makeidx}
\usepackage{geometry}
\usepackage{indentfirst}
\usepackage{enumitem}
\usepackage{hyperref}
\usepackage{mathrsfs}
\usepackage{enumitem}
\usepackage{verbatim}

\newtheorem{thm}{Theorem}[subsection]
\newtheorem{defn}[thm]{Definition}
\newtheorem{prop}[thm]{Proposition}
\newtheorem{lem}[thm]{Lemma}
\newtheorem{cor}[thm]{Corollary}

\AfterEndEnvironment{thm}{\noindent\ignorespaces}
\AfterEndEnvironment{defn}{\noindent\ignorespaces}
\AfterEndEnvironment{prop}{\noindent\ignorespaces}
\AfterEndEnvironment{lem}{\noindent\ignorespaces}
\AfterEndEnvironment{cor}{\noindent\ignorespaces}

\newenvironment{subproof}[1][\proofname]{%
  \begin{proof}[#1]%
}{%
  \end{proof}%
}

\newcommand{\ian}{i_{an}}
\newcommand{\xan}{X_{an}}
\newcommand{\rsp}{\operatorname{rsp}}

\newcommand{\F}{\mathscr{F}}
\newcommand{\G}{\mathscr{G}}
\newcommand{\Sh}{\operatorname{Sh}}

\newcommand{\hypc}{\mathbb{H}}
\newcommand{\Hom}{Hom}
\newcommand{\id}{\operatorname{id}}
\newcommand{\supp}{\operatorname{supp}}
\newcommand{\open}{\operatorname{Open}}

\newcommand{\algs}{\operatorname{Sh}_{alg}}
\newcommand{\ans}{\operatorname{Sh}_{an}}
\newcommand{\Coh}{\operatorname{Coh}}
\newcommand{\QCoh}{\operatorname{QCoh}}

\newcommand{\oalg}{\mathscr{O}_{X}}
\newcommand{\oan}{\mathscr{O}_{X_{an}}}
\newcommand{\han}{\mathscr{H}_{X_{an}}}

\title{Relative-Hyper GAGA Theorem}
\author{Taewan Kim, Eita Haibara}

\begin{document}
\maketitle

\begin{abstract}
In this paper, we provide a relative hypercohomology version of Serre's GAGA theorem.
We prove that the relative hypercohomology of a complex of sheaves on a complex projective variety
is isomorphic to the relative hypercohomology of its analytification, with respect to an open or closed subvariety.
This result implies Serre's original GAGA theorem.
\end{abstract}
\tableofcontents

\section{Introduction}
Serre's GAGA paper \cite{gaga} establishes relationships between the algebraic structure of $X$ and the analytic structure on $X_{an}$.
One of the main results in the paper is the following theorem:
\begin{thm}[{\cite[Th\'eor\`eme 1]{gaga}}]
    Suppose that $X$ is a complex projective variety.
    For every coherent algebraic sheaf $\F$ on $X$, there is an isomorphism
    $$H^*(X, \F)\simeq H^*(X_{an}, \F_{an}).$$
\end{thm}
This theorem naturally extended to hypercohomology, as Theorem \ref{thm:hypergaga}.

We have a definition of relative hypercohomology (see Definition \ref{def:relative_hyper}), 
using the functors of derived categories $Ri_!$ and $i^!$ that appear in the Verdier duality \eqref{homverdier}.
Since relative hypercohomology involves derived functors, it is not so clear anymore that the above theorem appears in relative hypercohomology.
However, a subvariety is a locally closed space in Zariski topology, thus we can simplify the derived functor.
We provide a relative hypercohomology version of the theorem in open subvarieties and closed subvarieties with certain conditions.
(the case of an open subvariety is Theorem \ref{thm:open_case}, the case of a closed subvariety is Theorem \ref{thm:closed_case}). 

\section{Background}
\subsection{Verdier Duality}
Let us denote $\Sh(X)$ as the category of sheaves over $X$, and let $f:X\rightarrow Y$ be a continuous map between topological spaces.
It induces the direct image functor $f_*:\Sh(X)\rightarrow \Sh(Y)$, defined by $f(\F)(U)=\F(f^{-1}(U))$ for $\F\in\Sh(X)$ and $U\in \open(X)$.
On the other hand, we have an inverse image functor $f^*:\Sh(Y)\rightarrow \Sh(X)$. 
Let $E\rightarrow Y$ be the \'etal\'e space of the sheaf $\F$ on $Y$.
There exist pullbacks in categories of topological spaces. Thus, we can pullback this \'etal\'e space and get the diagram below:
\begin{center}
    \begin{tikzcd}
        X\times_Y E \arrow[d] \arrow[r] & E \arrow[d] \\
        X \arrow[r, "f"']               & Y          
    \end{tikzcd}
\end{center}
Then the bundle $X\times_Y B\rightarrow X$ defines a sheaf on $X$, denoted $f^*(\G)$.
Indeed, an inverse image functor can be calculated as $f^*(\G)(U)=\varinjlim_{V\supset f(U)}\G(V)$.

They are adjoint to each other, i.e. they satisfy the following natural isomorphism
$$\Hom_{\Sh(X)}(f^*\G,\F)\simeq\Hom_{\Sh(Y)}(\G, f_*\F).$$
And we have the direct image functor with compact support $f_!:\Sh(X)\rightarrow\Sh(Y)$, defined by 
$$f_!(\F)(U)=\{s\in\F(f^{-1}(U)) \mid f|_{\supp(s)}:\supp(s)\rightarrow U \text{ is proper map}\}.$$
We expect that there exists a functor $f^!$ which is right adjoint to $f_!$. 
However, in a general setting, such a functor does not always exist.
Under certain conditions, there exists a functor $f^!:D^+\Sh(Y)\rightarrow D^+\Sh(X)$ which is right adjoint to $Rf_!$ in derived category.
This means that there is an isomorphism
\begin{equation}\label{globalverdier}
    \Hom_{D^+\Sh(X)}(Rf_!\F^\bullet, \G^\bullet) \simeq \Hom_{D^+\Sh(Y)}(\F^\bullet, f^!\G^\bullet).
\end{equation}
More generally, there is an isomorphism of derived hom-sheaves
\begin{equation}\label{homverdier}
    R\mathcal{H}om(Rf_!\F^\bullet, \G^\bullet) \simeq Rf_*R\mathcal{H}om(\F^\bullet, f^!\G^\bullet).
\end{equation}
We can get the global version \eqref{globalverdier} by applying $h^0\Gamma(Y, -)$ to \eqref{homverdier}.

If $X$ and $Y$ are locally compact Hausdorff (and a few additional conditions), the functor $f^!$ exists.
However, we work on varieties that have the Zariski topology, and they are not Hausdorff.
Fortunately, our interest lies in subvarieties that are locally closed subspace of the entire variety, and for which the functor $f^!$ exists.

\subsection{Duality of locally closed subspace}
This subsection recalls \cite[II.6]{iver}.
Let $Y$ be a locally closed subset of $X$, and $i:Y\hookrightarrow X$ be the natural inclusion.
\begin{defn}
    For a sheaf $\G\in \Sh(Y)$, we have a functor $i_!:\Sh(Y)\rightarrow\Sh(X)$ defined by
    $$i_!(\G)(U) = \{s\in\G(Y\cap U) \mid \supp(s)\text{ is closed relative to } U\}$$
    where $U\in \open(X)$. 
\end{defn}

To construct a functor $i^!$, first we have to construct a functor $\rsp_i$.
For a sheaf $\F\in \Sh(X)$, there exists the functor $\rsp_i:\Sh(X)\rightarrow\Sh(X)$ by
$$\rsp_i(\F)(V) = \{s\in\F(V) \mid \supp(s)\subset Y\}.$$
We sometimes write $\rsp_i(\F)$ as $\F^Y$.
Then we can construct the functor $i^!$ using $\rsp_i$:
\begin{defn}
    We have a functor $i^!:\Sh(X)\rightarrow \Sh(Y)$ defined by
    $$i^!:=i^*\rsp_i.$$
\end{defn}
We can see that these functors are adjoint to each other.
\begin{thm}[{\cite[II.6.3, II.6.8]{iver}}]\label{thm:local_duality}
    The two functors
    \begin{center}
        \begin{tikzcd}
            \Sh(Y) \arrow[r, "i_!"', shift right] & \Sh(X) \arrow[l, "i^!"', shift right]
        \end{tikzcd}
    \end{center}
    satisfy the following identity
    $$\Hom_{\Sh(X)}(i_!\G,\F)\simeq\Hom_{\Sh(Y)}(\G, i^!\F)$$
\end{thm}
Hence $i_!$ is the left adjoint and $i^!$ is the right adjoint. 

Suppose $\F$ is a sheaf on $X$ whose stalks vanish outside of $Y$. 
Then $i_!i^*\F\simeq\F$, and we get $i_! i^!\simeq \rsp_i$.
Thus we can describe the section of the functor $i_!i^!$ as 
$$i_! i^!(\G)(V) = \{s\in \G(V) \mid \supp(s)\subset Y \}$$
for a sheaf $\G\in \Sh(X)$.

\begin{prop}
    The functor $i_!$ is exact, and the functor $f^!$ is left exact.
\end{prop}
Hence, we can define the right derived functor $Rf^!$, and it is isomorphic to the functor $f^!$ in ordinary verdier duality \eqref{homverdier} (see \cite[V.7.19]{borel}).

There are exactness properties of certain functors in specific cases (see \cite[II.6.9 and II.6.12]{iver}):
\begin{prop}\label{prop:exactness}
    Suppose that $i:U\hookrightarrow X$ is an inclusion of an open subset and $j:Z\hookrightarrow X$ is an inclusion of a closed subset.
    \begin{itemize}
        \item $i^*=i^!$ and $j_*=j_!$.
        \item $i_!$, $j^*$, $i^*=i^!$ and $j_*=j_!$ are exact.
        \item $i_*$ and $j^!$ are left exact and preserving injectives.
    \end{itemize}
\end{prop}

And there is a locally closed space version of the base change theorem.
\begin{thm}\label{thm:basechange}
    Let us consider a continuous map $f:X\rightarrow Y$, and a locally closed subset $B\subset Y$. Note that a subspace $A=f^{-1}(B)\subset X$ is also locally closed.
    There is a commutative diagram
    \begin{center}
        \begin{tikzcd}
            A \arrow[r, "f'"] \arrow[d, "i"', hook] & B \arrow[d, "j", hook] \\
            X \arrow[r, "f"']                       & Y                     
        \end{tikzcd}
    \end{center}
    and the following isomorphisms hold:
    $$f^* j_! \F \simeq i_! f'^* \F$$
    for $\forall \F\in \Sh(B)$, and 
    $$j^! f_* \G \simeq f'_* i^! \G$$
    for $\forall \G\in \Sh(X)$.
\end{thm}

\subsection{Relative hypercohomology}
This section is according to \cite[5.5]{maxim}.
\begin{defn}
    Let $\F^\bullet$ be a complex of sheaves of $X$. 
    The hypercohomology of $\F^\bullet$ is a cohomology of the right derived functor of the global section functor:
    $$\hypc^*(X; \F^\bullet):=R^*\Gamma(X, \F^\bullet)=h^*R\Gamma(X, \F^\bullet)$$
\end{defn}

Relative hypercohomology is a hypercohomology that uses functors in Verdier duality:
\begin{defn}\label{def:relative_hyper}
    For an inclusion of topological spaces $f:Y\hookrightarrow X$, relative hypercohomology is defined by:
    $$\hypc^*(X, X-Y; \F^\bullet):=\hypc^*(X;Rf_!f^!\F^\bullet)$$
\end{defn}
If we use a local system, the classical Alexander duality comes from the definition of relative hypercohomology by using Verdier duality.

\subsection{Serre's GAGA}
In this paper, we use the definition of algebraic variety in \cite[n\textsuperscript{o}~34]{fac}. 
Note that every (quasi-)affine and (quasi-)projective variety is an algebraic variety.
To see the construction in this section more precisely, see \cite[\S 5]{gaga}.

Let $X$ be a complex algebraic variety. 
Since a locally closed subset of affine space can inherit the usual topology of an affine space $\mathbb{C}^n$, 
each chart of $X$ can inherit an induced topology of $\mathbb{C}^n$.
We denote $X_{an}$ as a topological space equipped with its underlying set $X$ and the induced topology.
Any open subset of Zariski topology of the affine space is also open in the usual topology. 
Thus, the usual topology is finer than the Zariski topology and there exists a natural continuous map $h:X_{an}\rightarrow X$.
Let us denote a structure sheaf of $X$ as $\oalg$. 
The pullback $h^*\oalg$ is a subsheaf of $\han$, which is a sheaf of germs of holomorphic functions on $X_{an}$.
We denote $h^*\oalg$ as $\oan$.

An algebraic sheaf (resp. analytic sheaf) is defined as a sheaf of $\oalg$-module on $X$ (resp. $\han$-module on $X_{an}$).
Let us denote the category of algebraic sheaves of $X$ as $\algs(X)$ and the category of analytic sheaves of $X_{an}$ as $\ans(X_{an})$.
We can define an analytification functor of sheaves $(-)_{an}:\algs(X)\rightarrow\ans(X_{an})$.

\begin{defn}\label{def:analytification}
    Suppose $\F$ is an algebraic sheaf of $X$.
    The continuous map $h:X_{an}\rightarrow X$ induces the functor $h^*:\algs(X)\rightarrow\algs(X_{an})$.
    Then we can construct the analytic sheaf $\F_{an} := h^*(\F)\otimes_{\oan}\han\in\ans(X_{an})$, and we call it an analytification of $\F$.
\end{defn}

\begin{thm}[{\cite[Proposition 10]{gaga}}]
    \leavevmode
    \begin{itemize}
        \item The functor $(-)_{an}$ is exact.
        \item If $\F$ is a coherent algebraic sheaf, then $\F_{an}$ is a coherent analytic sheaf.
    \end{itemize}
\end{thm}
The above result implies that the functor $(-)_{an}:\algs(X)\rightarrow\ans(X)$ can be restricted to the functor $(-)_{an}:\Coh_{alg}(X)\rightarrow\Coh_{an}(X)$ naturally.
\begin{thm}[{\cite[Th\'eor\`eme 2,3]{gaga}}]
    The functor $(-)_{an}$ is an equivalence of categories, i.e., fully faithful and essentially surjective. 
    More precisely, it satisfies the following statements:
    \begin{itemize}
        \item For every coherent algebraic sheaves $\F$ and $\G$ on $X$, the mapping of the hom-sets induced by analytification $\Hom_{\algs(X)}(\F, \G)\rightarrow\Hom_{\ans(X_{an})}(\F_{an}, \G_{an})$ is bijective.
        \item For every coherent analytic sheaf $\G$ on $X_{an}$, there exists a coherent algebraic sheaf $\F$ on $X$ such that $\F_{an}\simeq\G$.
    \end{itemize}
\end{thm}

Moreover, there is a relation between cohomologies of algebraic and analytic spaces.
\begin{thm}[{\cite[Th\'eor\`eme 1]{gaga}}]\label{thm:gaga}
    Suppose that $X$ is a complex projective variety (i.e. Zariski closed subvariety of a complex projective space $\mathbb{P}_r(\mathbb{C})$). 
    For every $\F\in\Coh_{alg}(X)$, there is an isomorphism
    $$H^*(X, \F)\simeq H^*(X_{an}, \F_{an}).$$
\end{thm}

\section{Main result}
Let $X$ be a complex projective variety, i.e., a closed subvariety of complex projective space, and let $Y$ be its subvariety. 
By the definition of subvariety, $Y$ is a locally closed subset of $X$ over its Zariski topology.
We have a natural embedding $i:Y\hookrightarrow X$, which is an injective continuous map.

First we provide a hypercohomology version of Theorem \ref{thm:gaga} as follows.
The proof is analogous to \cite[Theorem 2.6.1]{algdr}.
\begin{thm}[Hyper-GAGA]\label{thm:hypergaga}
    Suppose that $\F^\bullet$ is a complex of algebraic coherent sheaves on $X$, and $X_{an}$ (resp. $\F_{an}^\bullet$) is an analytification of $X$ (resp. $\F^\bullet$).
    Then the hypercohomology of $\F^\bullet$ on $X$ is isomorphic to the hypercohomology of $\F_{an}^\bullet$ on $X_{an}$:
    $$\hypc^*(X, \F^\bullet)\simeq\hypc^*(\xan, \F_{an}^\bullet)$$
\end{thm}
\begin{proof}
    According to Serre's GAGA, we know that $\hypc^p(X, \F^q)\simeq\hypc^p(\xan, \F_{an}^q)$ holds for $\forall p, q$.
    These induce isomorphisms at the $E_1$-stage of two spectral sequences of the hypercohomologies, thus we obtain an isomorphism at the $E_\infty$-stage $\hypc(X, \F^\bullet)\simeq\hypc(\xan, \F_{an}^\bullet)$.
\end{proof}

One might want that the following application of Serre's GAGA to be established:
\begin{equation}\label{eqn:expected1}
    \hypc^*(X, X-Y; \F^\bullet) \simeq \hypc^*(X_{an}, X_{an}-Y_{an}; \F_{an}^\bullet)
\end{equation}
where $\F$ is a complex of algebraic coherent sheaves. 
First, since we are working in the setting of a locally closed subspace \ref{thm:local_duality}, we have $i^!\simeq Ri^!$ by a slight abuse of notation. 
Hence, we will denote the right adjoint functor of $Ri_!$ as $Ri^!$, and denote the right adjoint functor of $i_!$ as $i^!$ in this section.
And we know that $i_!$ is exact by Proposition \ref{prop:exactness}, and thus $Ri_!\simeq i_!$ by \cite[10.5.2]{weibel}.
Therefore, what we want is the following relation:
\begin{equation}\label{eqn:expected2}
    \hypc^*(X; i_!Ri^!\F^\bullet) \simeq \hypc^*(X_{an}; i_!Ri^!\F_{an}^\bullet).
\end{equation}
To prove \eqref{eqn:expected2} using Theorem \ref{thm:hypergaga}, 
we have to show that $i_!Ri^!\F^\bullet$ is quasi-isomorphic to a complex of coherent sheaves and that there exists a quasi-isomorphism
\begin{equation}
    (i_!Ri^!\F^\bullet)_{an} \stackrel{qis}{\simeq} (i_{an})_!R(i_{an})^!\F^\bullet.
\end{equation}
Since $(-)_{an}$ is exact, it can be naturally extended to derived category, i.e. $(-)_{an}:D^+\algs(X)\rightarrow D^+\ans(X_{an})$.

Unfortunately, there are not enough injective objects in the category of coherent sheaves $Coh(X)$.
Thus, it is difficult to prove that the result of the derived functor is coherent 
and that analytification functor commutes with it.
Therefore, in order to obtain the result in \eqref{eqn:expected2}, we will consider certain situations and conditions.

Note that the category of quasi-coherent sheaves $\QCoh(X)$ has enough injectives, thus the two derived functors 
\begin{center}
    \begin{tikzcd}
        D^+\QCoh(Y) \arrow[r, "Ri_!"', shift right] & D^+\QCoh(X) \arrow[l, "Ri^!"', shift right]
    \end{tikzcd}
\end{center}
are well-defined. 
Hence, we can think of the derived categories that we are using as either the category of algebraic (resp. analytic) sheaves 
or the category of quasi-coherent algebraic (resp. analytic) sheaves.

\subsection{Open subvariety}
Let $Y$ be an open subvariety of $X$, and $i:Y\hookrightarrow X$ be the natural inclusion. 
In this case, we know that $i^!=i^*$ is exact (see Proposition \ref{prop:exactness}).
Since the derived functor of an exact functor is a trivial extension of original functor (see \cite[10.5.2]{weibel}), we get
\begin{equation}\label{eqn:opensub}
    Ri_! Ri^!\simeq i_! i^*.
\end{equation}

Let us denote $\F^Y:=i_!i^*\F$ and the restriction of $\F$ to $Z$ as $\F\rvert_Z:=i^*\F$. Then we have the following proposition.
\begin{prop}\label{prop:unique_restriction}
    Let $Z$ be a locally closed subspace. $\F^Z$ is the unique sheaf satisfying the following conditions:
    \begin{alignat*}{2}
        &\left.\F^Z\right\rvert_Z &&= \F\rvert_Z\\
        &\left.\F^Z\right\rvert_{X-Z} &&= 0.
    \end{alignat*}
\end{prop}
\begin{proof}
    See \cite[II.2.9.2]{godement}.
\end{proof}

We have the following theorem in this case.
\begin{thm}\label{thm:open_case}
    Let $X$ be a complex projective variety, $Y$ be an open subvariety of $X$, and $i:Y\hookrightarrow X$ be the natural inclusion.
    If $\F^\bullet$ is a complex of algebraic sheaves on $X$ such that $(\F^Y)^\bullet$ is a complex of coherent sheaves, then there is an isomorphism
    $$\hypc^*(X, X-Y; \F^\bullet)\simeq\hypc^*(X_{an}, X_{an}-Y_{an}; \F_{an}^\bullet).$$
\end{thm}
\begin{proof}
    Since an open subset of $X$ is also open in $X_{an}$, the natural inclusion of analytic spaces $i_{an}:Y_{an}\hookrightarrow X_{an}$ is also an inclusion of the open subspace.
    By the definition and \eqref{eqn:opensub},
    $$\hypc^*(X, X-Y; \F^\bullet) = \hypc^*(X; Ri_! Ri^!\F^\bullet) \simeq \hypc^*(X; i_! i^*\F^\bullet)$$
    and
    $$\hypc^*(X_{an}, X_{an}-Y_{an}; \F_{an}^\bullet) = \hypc^*(X_{an}; R(i_{an})_! R(i_{an})^!(\F_{an})^\bullet) \simeq \hypc^*(X_{an}; (i_{an})_! (i_{an})^*(\F_{an})^\bullet).$$
    If we can prove that
    \begin{equation}
        (i_!i^*\F^\bullet)_{an} \stackrel{qis}{\simeq} (i_{an})_! (i_{an})^*(\F_{an})^\bullet,
    \end{equation}
    then we can get the result by using Theorem \ref{thm:hypergaga} since $i_! i^*\F^\bullet=(\F^Y)^\bullet$ is algebraic coherent.
    We shall prove an isomorphism between the complexes, rather than the quasi-isomorphism.
    According to the definition of analytification \ref{def:analytification}, we have to prove
    \begin{equation}
        h^*(i_!i^*\F)\otimes_{\oan}\han \simeq (i_{an})_! (i_{an})^*(h^*\F\otimes_{\oan}\han).
    \end{equation}
    where $h:X_{an}\rightarrow X$ is the canonical continuous map.

    First, we can deduce that $i_!i^*$ and $h^*$ commute:
    \begin{lem}\label{lem:commute}
        For $\G\in\algs(X)$,
        $$h^*(i_! i^* \G) \simeq (\ian)_! (\ian)^* (h^*\G)$$
    \end{lem}
    \begin{subproof}[Subproof]
        We have a commutative diagram below:
        \begin{center}\label{diagram}
            \begin{tikzcd}
                Y_{an} \arrow[r, "h'"] \arrow[d, "i_{an}"', hook] & Y \arrow[d, "i", hook] \\
                X_{an} \arrow[r, "h"']                            & X                     
            \end{tikzcd}
        \end{center}
        where $h'$ is the restriction of $h$ to $Y_{an}$.
        Note that $Y_{an}$ is equal to $h^{-1}(Y)$.
        By Theorem \ref{thm:basechange}, we have 
        $$h^* i_!=(\ian)_! h'^*.$$
        Thus, we obtain
        \begin{align}
            \begin{split}
                h^* i_! i^* &= (\ian)_! h'^* i^*\\
                            &= (\ian)_! (i h')^* \\
                            &= (\ian)_! (h \ian)^*\\
                            &= (\ian)_! (\ian)^* h^*.
            \end{split}
        \end{align}
    \end{subproof}
    It remains to show the isomorphism below:
    \begin{equation}\label{eqn:rest_anal}
        (h^*\F)^{Y_{an}}\otimes_{\oan}\han\simeq \left((h^*\F)\otimes_{\oan}\han\right)^{Y_{an}}
    \end{equation}

    \begin{lem}
        Let $(X, \mathscr{O}_X)$ be a ringed space and $Z$ be its locally closed subspace with canonical inclusion.
        If $\G$ is a sheaf of $\mathscr{O}_X$-module, then there is an isomorphism
        $$\G^Z\simeq\G\otimes_{\mathscr{O}_X}\mathscr{O}_X^Z$$
    \end{lem}
    \begin{subproof}[Subproof]
        Using Proposition \ref{prop:unique_restriction}, we have to prove that $\G\otimes_{\mathscr{O}_X}\mathscr{O}_X^Z$ is also satisfying the conditions:
        \begin{equation}\label{eq:wts_cond}
            \begin{alignedat}{2}
                    &\left.\left(\G\otimes_{\mathscr{O}_X}\mathscr{O}_X^Z\right)\right\rvert_Z &&= \G\rvert_Z\\
                    &\left.\left(\G\otimes_{\mathscr{O}_X}\mathscr{O}_X^Z\right)\right\rvert_{X-Z} &&= 0.
            \end{alignedat}
        \end{equation}
        Let $i:Z\hookrightarrow X$ and $j:X-Z\hookrightarrow X$ be the natural inclusions.
        We can deduce
        \begin{align}
            \begin{split}
                \left.\left(\G\otimes_{\mathscr{O}_X}\mathscr{O}_X^Z\right)\right\rvert_Z
                &= i^*\left(\G\otimes_{\mathscr{O}_X}\mathscr{O}_X^Z\right)\\
                &= i^*\G \otimes_{\mathscr{O}_Z} i^*\mathscr{O}_X^Z\\
                &= i^*\G \otimes_{\mathscr{O}_Z} i^*\mathscr{O}_X\\
                &= i^*\left(\G \otimes_{\mathscr{O}_X} \mathscr{O}_X\right)\\
                &= i^*\G\\
                &= \G\rvert_Z.
            \end{split}
        \end{align}
        Similary, we obtain
        \begin{align}
            \begin{split}
                \left.\left(\G\otimes_{\mathscr{O}_X}\mathscr{O}_X^Z\right)\right\rvert_{X-Z}
                &= j^*\G \otimes_{\mathscr{O}_{X-Z}} j^*\mathscr{O}_X^Z\\
                &= j^*\G \otimes_{\mathscr{O}_{X-Z}} 0\\
                &= 0.
            \end{split}
        \end{align}
        Therefore, the conditions \eqref{eq:wts_cond} are satisfied.
    \end{subproof}

    \begin{lem}\label{lem:tensor_commute}
        For $\G\in\algs(X_{an})$,
        $\G^{Y_{an}}\otimes_{\oan}\han\simeq \left(\G\otimes_{\oan}\han\right)^{Y_{an}}$
    \end{lem}
    \begin{subproof}[Subproof]
        We have 
        $$\G^{Y_{an}}\otimes_{\oan}\han\simeq \left(\G\otimes_{\oan}(\oan)^{Y_{an}}\right)\otimes_{\oan}\han$$ 
        and 
        $$\left(\G\otimes_{\oan}\han\right)^{Y_{an}} \simeq \left(\G\otimes_{\oan}\han\right)\otimes_{\oan}(\oan)^{Y_{an}}.$$
        Thus, the result follows from the associativity of the tensor product."
    \end{subproof}

    According to the above lemmas, we get an isomorphism of \eqref{eqn:rest_anal}.
\end{proof}

\begin{cor}
    Theorem \ref{thm:open_case} implies the original GAGA \ref{thm:gaga} and the Hyper-GAGA \ref{thm:hypergaga}.
\end{cor}
\begin{proof}
    Suppose $Y=X$ and $i=\id_X$.
    Then we have $Ri_!Ri^!\simeq i_!i^!\simeq \rsp_{\id_X} = \id_{\algs(X)}$.
    Applying Theorem \ref{thm:open_case} with the condition that $(\F^Y)^\bullet=\F^\bullet$ is coherent, we get
    $$\hypc^*(X; \F^\bullet)\simeq\hypc^*(X_{an}; \F_{an}^\bullet).$$
    This implies \ref{thm:hypergaga} and also \ref{thm:gaga}.
\end{proof}

\subsection{Closed subvariety}
Let $Y$ be a closed subvariety of a complex projective variety $X$ and $i:Y\hookrightarrow X$ be a canonical inclusion. 
In this case, we know that $i_!=i_*$ is exact (see Proposition \ref{prop:exactness}). Thus we obtain
\begin{equation}
    Ri_! Ri^!\simeq i_* Ri^!.
\end{equation}

We will give an acyclic condition to remove the derived functor.
The following lemma is originally provided in \cite[2.4]{tohoku}, and we give a proof with reference \cite[Theorem 2.4.2]{algdr}.
\begin{prop}\label{prop:acyclic}
    Let $\F^\bullet$ be a complex of $i^!$-acyclic sheaves. 
    Then the value of the derived functor $Ri^!$ of $\F^\bullet$ is isomorphic to the value of $i^!$ in the derived category $D^+\algs(X)$:
    $$Ri^!(\F^\bullet)\simeq i^!\F^\bullet$$
\end{prop}
\begin{proof}
    Let $K^{\bullet, \bullet}$ be a double complex where $K^{\bullet, 0}$ is $\F^\bullet$, and $K^{p, \bullet}$ is an injective resolution of $\F^p$.
    The $i^!$-acyclic condition means that the $E_1$-stage of the spectral sequence of $K^{\bullet, \bullet}$, which is induced by $i^!$, is
    \begin{align*}
        E_1^{p, q}=h^p Ri^! (K^{\bullet, q})=
        \begin{cases}
            i^!\F^q &\text{ if }p=0\\
            0   &\text{ if }p>0.
        \end{cases}
    \end{align*}

    \begin{center}
        \begin{tikzpicture}
            \matrix (m) [matrix of math nodes,
            nodes in empty cells, nodes={minimum width=4ex,
            minimum height=3ex, outer sep=-1pt},
            column sep=1ex,row sep=1ex]{
                        q   &         &     &     &\\
                        2   & i^!\F^2 &  0  &  0  &\\
                        1   & i^!\F^1 &  0  &  0  &\\
                        0   & i^!\F^0 &  0  &  0  &\\
                        E_1 & 0 &  1  &  2  & p\strut \\};
            \draw[<-, thick] (m-1-1.east) -- (m-5-1.east) ;
            \draw[->, thick] (m-5-1.north) -- (m-5-5.north) ;
        \end{tikzpicture}
    \end{center}
    Thus, the $E_2$-stage is 
    \begin{align*}
        E_2^{p, q}=h^{p,q}E_1 =
        \begin{cases}
            h^q (i^!\F^\bullet) &\text{ if }p=0\\
            0   &\text{ if }p>0.
        \end{cases}
    \end{align*}
    It means $h^kRi^!(\F^\bullet)\simeq h^k(i^!\F^\bullet)$. Thus we get 
    $$Ri^!(\F^\bullet) \stackrel{qis}{\simeq} i^!\F^\bullet.$$
\end{proof}

In this case of a closed subvariety, we need more conditions than the open subvariety case.
\begin{thm}\label{thm:closed_case}
    Let $X$ be a complex projective variety, $Y$ be its closed subvariety, and $i:Y\hookrightarrow X$ be a natural inclusion.
    Suppose $\F^\bullet$ is a complex of locally free $i^!$-acyclic algebraic sheaves on $X$, $h^*\F^\bullet$ is flasque, and $i_*i^!\F^\bullet$ is quasi-isomorphic to a complex of coherent sheaves. 
    Then there is an isomorphism
    $$\hypc^*(X, X-Y; \F^\bullet)\simeq\hypc^*(X_{an}, X_{an}-Y_{an}; \F_{an}^\bullet).$$
\end{thm}
\begin{proof}
    We will use the similar method as the open subvariety case.
    By the above Proposition \ref{prop:acyclic}, we can use $i^!\F^\bullet$ instead of $Ri^!\F^\bullet$ by the condition.
    We shall prove that the following isomorphism holds:
    \begin{equation}\label{eqn:closed_purpose}
        (i_* i^!\F^\bullet)_{an}\simeq(i_{an})_*(i_{an})^!\F_{an}^\bullet
    \end{equation}
    It suffices to show the existence of the isomorphism \eqref{eqn:closed_purpose} for the single sheaf $\F$.
    
    By the definition of analytification, \eqref{eqn:closed_purpose} is 
    \begin{equation}
        h^*(i_* i^!\F)\otimes_{\oan}\han \simeq (i_{an})_*(i_{an})^!(h^*\F\otimes_{\oan}\han)
    \end{equation}
    where $h$ is in Definition \ref{def:analytification}.
    Using Lemma \ref{lem:commute}, we get
    \begin{align}\label{eqn:commute_closed}
        \begin{split}
            h^* i_* i^! \F
            &= h^* i_! i^* \rsp_i \F \\
            &\simeq (i_{an})_! (i_{an})^* h^* \rsp_i \F.
        \end{split}
    \end{align}
    Suppose $U$ is an open subset of $X$ and $U_{an}\subset X_{an}$ is its analytic space. Then we have
    \begin{align}
        \begin{split}
            (h^* \rsp_i \F)(U_{an})
            &=(\rsp_i\F)(U)\\
            &=\left\{s\in \F(U)\mid \supp(s)\subset Y\right\}\\
            &=\left\{s'\in h^*\F(U_{an})\mid \supp(s')\subset Y_{an}\right\}\\
            &=(\rsp_{i_{an}}h^*\F)(U_{an}).
        \end{split}
    \end{align}
    It follows that $h^* \rsp_i = \rsp_{i_{an}} h^*$. Thus the last term of \eqref{eqn:commute_closed} is
    $$(i_{an})_! (i_{an})^* \rsp_{i_{an}} h^* \F = (i_{an})_* (i_{an})^! h^* \F,$$
    then we get
    \begin{equation}
        h^* i_* i^! \F \simeq (i_{an})_* (i_{an})^! h^* \F.
    \end{equation}

    Now we shall show that 
    \begin{equation}\label{eqn:wts_closed}
        (i_{an})_* (i_{an})^! \F'\otimes_{\oan}\han \simeq (i_{an})_*(i_{an})^!(\F'\otimes_{\oan}\han)
    \end{equation}
    where we denote $\F'=h^*\F\in\algs(X_{an})$.
    To do this, we give two lemmas.

    \begin{lem}\label{lem:tensor_expression}
        Let $(X, \oalg)$ be a locally ringed space, $Z$ be its subspace of $X$, and $i:Z\hookrightarrow X$ be natural inclusion.
        If $\G$ is a locally free $\oalg$-sheaf on $X$, then
        $$i_*i^*\G\simeq \G\otimes_{\oalg} i_*\mathscr{O}_{Z}$$
    \end{lem}
    \begin{subproof}
        By the projection formula (see \cite[Exercise II.5.1 (d)]{hart}),
        $$i_*(i^*\G\otimes_{\mathscr{O}_{Z}}\mathscr{O}_{Z})\simeq \G\otimes_{\oalg}i_*\mathscr{O}_{Z}.$$
        The RHS is $i_*(i^*\G\otimes_{\mathscr{O}_{Z}}\mathscr{O}_{Z})\simeq i_* i^*\G$, and this proves our lemma.
    \end{subproof}

    The following lemma forms the core of this proof.
    \begin{lem}
        Let $Y$ is closed subvariety of $X$ and $i:Z\hookrightarrow X$ be the natural inclusion.
        Suppose $\G$ is a locally free flasque algebraic sheaf on $X_{an}$. Then
        $$(i_{an})_* (i_{an})^!(\G\otimes_{\oan} \han) \simeq (i_{an})_* (i_{an})^!\G\otimes_{\oan}\han$$
    \end{lem}
    \begin{subproof}
        Let $j:X-Y\hookrightarrow X$ denote the natural inclusion of an open subvariety.
        Since $Y_{an}$ is closed and $\G$ is flasque, the natural mapping $\Gamma(U_{an}, \G)\rightarrow\Gamma(U_{an}-Y_{an}, \G)$ is surjective.
        By \cite[5.6]{maxim}, there exist a short exact sequence
        \begin{equation}\label{eqn:short_exact}
            0\longrightarrow (i_{an})_* (i_{an})^!\G \longrightarrow \G \longrightarrow (j_{an})_* (j_{an})^*\G \longrightarrow 0
        \end{equation}
        We can make two short exact sequences from \eqref{eqn:short_exact}.
        First, we substitute $\G\otimes_{\oan} \han$ for $\G$ in the sequence.
        Second, since $\han$ is $\oan$-flat sheaf by \cite[Corollaire 1]{gaga} and \cite[Proposition 22]{gaga}, 
        we can apply $-\otimes_{\oan}\han$ for each term of the sequence.
        The diagram below shows that these two short exact sequences coincide at the center.
        \begin{center}
            \begin{tikzcd}[column sep=2.6ex]
                0 \arrow[r] & (i_{an})_* (i_{an})^!(\G\otimes_{\oan} \han) \arrow[r] \arrow[d, no head, dashed] & \G\otimes_{\oan} \han \arrow[r]\arrow[d, no head, double] & (j_{an})_* (j_{an})^* (\G\otimes_{\oan} \han) \arrow[d, no head, dashed] \arrow[r]  & 0 \\
                0 \arrow[r] & ((i_{an})_* (i_{an})^!\G)\otimes_{\oan}\han  \arrow[r]                            & \G\otimes_{\oan} \han \arrow[r]                           & ((j_{an})_*(j_{an})^*\G)\otimes_{\oan}\han  \arrow[r] & 0
            \end{tikzcd}
        \end{center}
        \bigskip
        Thus, if we prove that 
        \begin{equation}\label{eqn:sub_isomorphic}
            (j_{an})_* (j_{an})^* (\F\otimes_{\oan} \han) \simeq \left((j_{an})_*(j_{an})^*\F\right)\otimes_{\oan}\han,
        \end{equation}
        we can get 
        \begin{equation}
            (i_{an})_* (i_{an})^!(\F\otimes_{\oan} \han) \simeq \left((i_{an})_* (i_{an})^!\F\right)\otimes_{\oan}\han.
        \end{equation}
        From Lemma \ref{lem:tensor_expression}, we have the following expressions using tensor product:
        $$(j_{an})_* (j_{an})^* (\F\otimes_{\oan} \han)\simeq (\F\otimes_{\oan} \han)\otimes_{\oan}(j_{an})_*\mathscr{O}_{X_{an}-Y_{an}}$$
        and 
        $$\left((j_{an})_*(j_{an})^*\F\right)\otimes_{\oan}\han \simeq (\F\otimes_{\oan}(j_{an})_*\mathscr{O}_{X_{an}-Y_{an}})\otimes_{\oan}\han.$$
        Like the case of open subvarity, the associativity of tensor product prove the isomorphism.
    \end{subproof}
    By the above lemma and the condition that $h^*\F$ is locally free and flasque, we get the isomorphism in \eqref{eqn:wts_closed}, and it finishes the proof.
\end{proof}


\newpage
\bibliographystyle{alpha}
\bibliography{ref}

\large{
    \bigskip   
    \bigskip
    \bigskip

    \textsl{Taewan Kim}\\
    \textit{E-mail:}
    \href{mailto:taewankim0109@gmail.com}{\nolinkurl{taewankim0109@gmail.com}}
    
    \bigskip
    \bigskip

    \textsl{Eita Haibara}\\
    \textit{E-mail:}
    \href{mailto:eita_haibara@protonmail.com}{\nolinkurl{eita_haibara@protonmail.com}}

    \bigskip
    \bigskip

    \textit{Visiting our page!}\\
    \large{\url{https://sites.google.com/view/pocariteikoku/home}}
}

\end{document}